\documentclass[12pt]{article}

\usepackage[margin=1in]{geometry}
\usepackage{float}

\usepackage{relsize}

\usepackage{diagbox}
\usepackage{mathtools}
\usepackage{bbm}
\usepackage{latexsym}
\usepackage{epsfig}
\usepackage{amsmath,amsthm,amssymb,enumerate}

\allowdisplaybreaks
\def\bZ{{\bf Z}}
\usepackage{color}

\def\cN{{\mathcal N}}

\def\nn{\nonumber}
\def\a{\alpha}  \def\d{\delta} \def\D{\Delta}
\def\e{\varepsilon}    
\def\G{\Gamma}  
 \def\th{\theta}    
  \def\n{\nu} \def\p{\pi}
\def\r{\rho}  \def\s{\sigma} 
 \def\om{\omega}  \def\U{\Upsilon}

\renewcommand{\c}[1]{{\mathcal #1}}
\def\cP{{\cal P}}

\def\bx{{\bf x}}\def\by{{\bf y}}

\newtheorem{theorem}{Theorem}
\newtheorem{lemma}[theorem]{Lemma}
\newtheorem{corollary}[theorem]{Corollary}

\newtheorem{claim}[theorem]{Claim}

\newtheorem{proposition}{Proposition}

\newcommand{\proofend}{\hspace*{\fill}\mbox{$\Box$}}

\newcommand{\Okl}[1]{\Omega_1^{(#1)}}

\newcommand{\wh}[1]{\widehat{#1}}

\newcommand{\rdup}[1]{{\left\lceil #1\right\rceil }}
\newcommand{\rdown}[1]{{\left\lfloor #1\right \rfloor}}

\newcommand{\brac}[1]{\left(#1\right)}

\newcommand{\bfrac}[2]{\left(\frac{#1}{#2}\right)}

\def\cE{{\cal E}}

\newcommand{\set}[1]{\left\{#1\right\}}

\def\E{\mathbb{E}}
\def\Var{\mathbb{V}ar}
\def\Pr{\mathbb{P}}

\newcommand{\ignore}[1]{}

\def\cC{{\mathcal C}}
\def\cD{{\mathcal D}}
\def\cE{{\mathcal E}}

\def\cM{{\mathcal M}}
\def\cN{{\mathcal N}}

\def\cP{{\mathcal P}}

\newcommand{\card}[1]{\left|#1\right|}

\newcommand{\beq}[2]{\begin{equation}\label{#1}#2\end{equation}}
\newcommand{\mults}[1]{\begin{multline*}#1\end{multline*}}
\newcommand{\mult}[2]{\begin{multline}\label{#1}#2\end{multline}}

\def\nn{\nonumber}

\def\cN{{\cal N}}
\def\bd{{\bf d}}

\def\bz{{\bf z}}

\def\leb{\leq_b}
\newcommand{\dnm}[1]{D_{n,m}^{(\d\geq#1)}}
\newcommand{\cdnm}[1]{{\mathcal D}_{n,m}^{(\delta\geq#1)}}
\newcommand{\gnm}[1]{G_{n,m}^{(\d\geq#1)}}

\usepackage{tikz}
\usetikzlibrary{decorations.pathmorphing}
\usetikzlibrary{positioning}
\usetikzlibrary{arrows,automata}
\usetikzlibrary{shapes.misc}
\usetikzlibrary{backgrounds}
\usetikzlibrary{arrows,shapes}

\begin{document}
\author{Colin Cooper\thanks{Research supported  at the University of Hamburg, by a DFG Mercator fellowship from Project 491453517}\\Department of Informatics\\
King's College\\
London WC2B 4BG\\England
\and
Alan Frieze\thanks{Research supported in part by NSF grant DMS1952285}\\Department of Mathematical Sciences\\Carnegie Mellon University\\Pittsburgh PA 15213\\U.S.A.}

\title{Hamilton cycles in random digraphs with minimum degree at least one}
\maketitle
\begin{abstract}
We study the existence of a directed Hamilton cycle in random digraphs with $m$ edges where we condition on minimum in- and out-degree at least one. Denote such a random graph by $D_{n,m}^{(\delta\geq1)}$. We prove that if $m=\tfrac n2(\log n+2\log\log n+c_n)$ then
\[
\lim_{n\to\infty}\Pr(D_{n,m}^{(\delta\geq1)}\text{ is Hamiltonian})=\begin{cases}0&c_n\to-\infty.\\e^{-e^{-c}/4}&c_n\to c.\\1&c_n\to\infty.\end{cases}
\]
\end{abstract}
\section{Introduction}
Let $D_{n,m}$ denote the random digraph with vertex set $[n]$ and $m$ random edges. McDiarmid \cite{McD} proved that if $m=n(\log n+\log\log n+\om)$ where $\om\to\infty$ then $D_{n,m}$ is Hamiltonian w.h.p. Subsequently, Frieze \cite{F1} sharpened this and proved that if $m=n(\log n+c_n)$
\beq{H1}{
\lim_{n\to\infty}\Pr(D_{n,m}\text{ is Hamiltonian})=\begin{cases}0&c_n\to-\infty.\\e^{-2e^{-c}}&c_n\to c.\\1&c_n\to\infty.\end{cases}
}
The R.H.S. of \eqref{H1} is the limiting probability that $\min\set{\d^-,\d^+}\geq 1$, where $\d^-$ (resp. $\d^+$) denotes the minimum in-degree (resp. out-degree) of $D_{n,m}$. In this paper we study the effect of conditioning on  $\min\set{\d^-,\d^+}$ being at least one. So, let $\cdnm{1}$ denote the set of digraphs with vertex set $[n]$ and $m$ edges such that $\min\set{\d^-,\d^+}\geq 1$. Then let $\dnm{1}$ be sampled uniformly from $\cdnm{1}$.
\begin{theorem}\label{th1}
Let $m=\tfrac n2(\log n+2\log\log n+c_n)$ then
\beq{H2}{
\lim_{n\to\infty}\Pr(\dnm{1}\text{ is Hamiltonian})=\begin{cases}0&c_n\to-\infty.\\e^{-e^{-c}/8}&c_n\to c.\\1&c_n\to\infty.\end{cases}
}
\end{theorem}
The R.H.S. of \eqref{H2} is the limiting probability that $\dnm{1}$ contains two vertices of in-degree one (resp. out-degree one) that share a common in-neighbour (resp. common out-neighbour).
\subsection{Previous work}
Previous work in this area has focused on $\gnm{k}$, where we condition on $G_{n,m}$ having minimum degree at least $k$. This work was inspired by the result of Koml\'os and Szemer\'edi \cite{KS}: suppose that $m=\frac{n}{2}(\log n+\log\log n+c_n)$. Then,
\begin{align*}
\lim_{n\to\infty}\Pr(G_{n,m}\text{ is Hamiltonian})&=\begin{cases}0&c_n\to-\infty.\\e^{-e^{-c}}&c_n\to c.\\1&c_n\to\infty.\end{cases}\\
&=\lim_{n\to\infty}\Pr(G_{n,m}\text{ has minimum degree at least two}).
\end{align*}
Bollob\'as, Fenner and Frieze \cite{BFF} proved the following. Suppose that  $m=\frac{n}{6}\brac{\log n+6\log\log n+c_n}$. Then
\beq{H3}{
\lim_{n\to\infty}\Pr(\gnm{2}\text{ is Hamiltonian})=\begin{cases}0&c_n\to-\infty,\text{ sufficiently slowly}.\\e^{-e^{-c}/1458}&c_n\to c.\\1&c_n\to\infty.\end{cases}
}
Here the R.H.S. of \eqref{H3}  is the asymptotic probability of there being three vertices of degree two sharing a common neighbour. The restriction ``sufficiently slowly'' is unlikely to be necessary as stated. On the other hand, there needs to be some upper bound on $c_n$ though as we must have $m\geq n$.

Beginning with Bollob\'as, Cooper, Fenner and Frieze \cite{BCFF} there has been an attempt to find out how many random edges we need if we condition on minimum degree at least three.
Anastos \cite{A} has proved that $\gnm{4}$ is Hamiltonian w.h.p. if $m>2n$ and Anastos and Frieze \cite{AF2} have shown that $\gnm{3}$ is Hamiltonian w.h.p. if $m>2.663n$. Because a random cubic graph is Hamiltonian w.h.p. (see Robinson and Wormald \cite{RW1}) it is natural to conjecture that $\gnm{3}$ is Hamiltonian w.h.p. if $m=cn, c\geq 3/2$.

It should be mentioned that the above papers \cite{A}, \cite{BCFF} prove more than Hamiltonicity. They show that in $\gnm{k}$ there are likely to be $\rdown{k/2}$ edge disjoint Hamilton cycles plus an edge disjoint perfect matching if $k$ is odd. It would not be difficult to prove a similar extension of Theorem \ref{th1} as no really new ideas are needed and so we will refrain from doing this, to avoid making the proof more complicated than needed.

\section{Proof of Theorem \ref{th1}}
We will base our proof on a 3-phase method  used in various forms in \cite{CF1}, \cite{DF}, \cite{FKR}, \cite{FS}. In Phase 1 we construct a set of $O(\log n)$ vertex disjoint directed cycles that cover all vertices. In Phase 2, we transform this cover so that each cycle has length at least $n/\log^{1/2}n$. In Phase 3, we {\em patch} these cycles together to create a Hamilton cycle. There is also a Phase 0 in which we partition the edge set into $E_1,E_2,E_3$ for use in the corresponding phase.

We first need to describe a useable model for $\dnm{1}$. For a sequence $\bx=(x_1,x_2,\ldots,x_{2m})\in [n]^{2m}$ we let $D_\bx$ be the multi-digraph with vertex set $[n]$ and edge set\\ $E_\bx=\set{(x_{2j-1},x_{2j}):j=1,2,\ldots,m}$. Let  $d^+_{\bx}(i)=|\set{j:x_{2j-1}=i}|$ be the out-degree of $i\in[n]$ in $D_\bx$ and similarly let  $d^-_{\bx}(i)=|\set{j:x_{2j}=i}|$ be the in-degree of $i\in[n]$ in $D_\bx$.  Then  $\Omega_1=\set{\bx\in [n]^{2m}:d^+_{\bx}(i),d^-_{\bx}(i)\geq 1}$ for all $i\in[n]$.

Given values for the degrees $d_\bx^{\pm}$ we create a random member of $\Omega_1$ by taking a random permutation of the multi-set $\set{d_{\bx}^+(j)\times j:j\in [n]}$ and placing the values in $x_1,x_3,\ldots,x_{2m-1}$ and then taking a random permutation of the multi-set $\set{d_{\bx}^-(j)\times j:j\in [n]}$ and placing the values in $x_2,x_4,\ldots,x_{2m}$. (The notation $d\times a$ means that the multi-set contains $d$ copies of $a$).

For $\bx\in\Omega_1$, let $L_\bx=\set{j\in[m]:x_{2j-1}=x_{2j}}$ be the set of loops in $D_\bx$,  let $M_\bx=\{j\in[m]:\exists j'\neq j\text{ s.t. }(x_{2j'-1},x_{2j'})=$ $(x_{2j-1},x_{2j})\}$ define the set of multiple edges in $D_\bx$; and let $\D_\bx=\max_j(d_{\bx}^-(j)+d_{\bx}^+(j))$. (We show later, see Lemma \ref{LM} below, that if $m\leq n\log n$ then w.h.p. there are severe restrictions on the closeness and multiplicity of loops and multiple edges.)

Let $\Omega_1^*=\set{\bx\in \Omega_1:L_\bx=M_\bx=\emptyset}$. We have that
\beq{sizeD}{
D\in \cdnm{1}\text{ implies that }|\set{\bx\in \Omega^*_1:D_\bx=D}|=m!.
}
Consequently, choosing $\bx$ (near) uniformly from $\Omega_1^*$ and taking $D_\bx$ is equivalent to sampling (near) uniformly from $\cdnm{1}$. By sampling uniformly from $\Omega_1$ and then using {\em switchings} to remove any loops or parallel edges we obtain a space $\Okl{0,0}$ which is almost all of $\Omega_1^*$  and whose entries are (near) uniformly distributed $\Okl{0,0}$.

We can generate a random member of $\Omega_1$ as follows: Let $Z=Z(z)$ be a {\em truncated Poisson} random variable with parameter $z\geq 1$, where
$$\Pr(Z=k)=\frac{z^k}{k!(e^z-1)},\hspace{1in}k=1,2,\ldots\ .$$
Here $z$ satisfies
\begin{equation}\label{2}
\E(Z)=\frac{ze^z}{e^z-1}=\r \text{ where }\r=\frac{m}{n}, \quad\text{ and }\quad \Var(Z)=\s^2=\frac{ze^z(e^z-1-z)}{(e^z-1)^2}.
\end{equation}
Note that 
\beq{zrho}{
z\sim\r\text{ if }\r\to\infty. 
}
Indeed,
\[
z\leq \r=z+\frac{1}{1+\frac{z^2}{2!}+\frac{z^2}{3!}+\cdots+}\leq z+1.
\]
Note also that $f(x)=xe^x/(e^x-1)$ is strictly monotone increasing and $f(0)=1,f(\infty)=\infty$ and so the first equation defines $z$ uniquely. This follows from $(e^x-1)^2f'(x)=e^{2x}-(1+x)e^x>0$.

 and note 
\begin{lemma}\label{Om1}
Let $Z_i',Z_i'',i\in [n]$ be independent copies of a truncated Poisson random variable $Z(z)$.
Let $\bx$ be chosen randomly from $\Omega_1$. Then $\{d_\bx^+(i)\}_{j\in [n]}$ is distributed as $\{Z_i'\}_{i\in [n]}$ conditional on $\sum_{j\in [n]}Z_i'=m$ and similarly for $\{d_\bx^-(i)\}_{i\in [n]}$ with respect to $Z_i''$.
\end{lemma}
The proof of this lemma is very similar to Lemma 4 of  Aronson, Frieze and Pittel \cite{AFP}, the proof is given in Appendix A.

We define the events
\beq{Edash}{
\cE'=\set{\sum_{j\in [n]}Z_i'=m}\text{ and }\cE''=\set{\sum_{j\in [n]}Z_i''=m}.
}
Lemma \ref{LM} (see Appendix B) proves some properties $\cP$ of $\Omega_1$ that hold w.h.p.  Let $\Okl{0,0}= \Omega_1^* \cap \cP$ and more generally let $\Okl{k,\ell}=\set{\bx\in \Omega_1:|L_\bx|=k,|M_\bx|=\ell\text{ and }\cP}$.

Given $\bx\in \Okl{k,\ell}$ with $\ell>0$, we will define a $P$-switch that removes two copies of edge $(x,y)$ and replaces them by loops $(x,x),(y,y)$. More formally, suppose that $i,j\in M_\bx$. Then a $P$-switch produces $\bx'$ where we make the replacements $x'_{2i}\gets x_{2j-1},x'_{2j-1}\gets x_{2i}$. After $\ell$ such $P$-switches we will have produced a member of $\Okl{k+2\ell,0}$.

Given $\bx\in \Okl{k,0}$ with $k>0$, we will define an $L$-switch that removes a loop $(x,x)$ by choosing an edge $(a,b)$ and replacing the pair by $(a,x),(x,b)$. More formally, suppose that $i\in L_\bx$. Let $j\notin L_\bx$. Then an $L$-switch produces $\bx'$ where we make the replacements $x'_{2i-1}\gets x_{2j-1},x'_{2j-1}\gets x_{2i-1}$.  After $k$ $P$-switches we will have produced a member of $\Okl{0,0}$. (The above can make new parallel edges, but they are few in number and we show in Lemma \ref{far} that we can exclude switches that create new parallel edges.)

$P$-switches preserve the precise degree sequence $d_{\bx}^+(i),d_{\bx}^-(i),i\in[n]$ for $\bx\in \Omega_1$ but $L$-switches do not. On the other hand, $L$-switches preserve the total degrees $d_{\bx}^+(i)+d_{\bx}^-(i),i\in[n]$ for $\bx\in \Omega_1$. However, they do preserve $d_{\bx}^+(i)\geq 1,d_{\bx}^-(i)\geq 1$ w.h.p. for $i\in[n]$.

We make the following claim: for proof see Appendix B.
\begin{claim}\label{cl1}
Suppose that $\bx$ is a random member of $\Omega_1$ and that $|L_\bx|=k,|M_\bx|=\ell$. Then after $\ell$ $P$-switches and $k+2\ell$ random $L$-switches we obtain $\by\in \Okl{0,0}$ such that (i) \by\ is almost uniform in $\Okl{0,0}$ and (ii) $|\Okl{0,0}|= (1-o(1))|\Omega_1^*|$. (By almost uniform in (i) we mean the following: if \bz\ is any member of $\Okl{0,0}$, then $\Pr(\by=\bz)=(1+o(1))|\Okl{0,0}|^{-1}$.)
\end{claim}
By a random $L$-switch, we mean choose one of the $k$ loops and then choose a random edge $(a,b)$. Lemma \ref{LM} below implies that w.h.p. we need at most $3e^2\log^4n$ switches altogether.

With a little bit of effort, we can use the analysis for Claim \ref{cl1} to obtain an estimate of the size of $\cdnm{1}$, see Appendix C.
\begin{theorem}\label{th3}
\[
|\cdnm{1}|\sim m!\frac{(e^z-1)^{2n}e^{-z(z+1)}}{2\p \s z^{2m}}.
\]
\end{theorem}
\subsection{Phase 0}
There is a natural bijection between digraphs and bipartite graphs. Given a digraph $D$ on vertex set $[n]$, we can define a bipartite graph $G(D)$ with vertex partition $=\set{a_i:i\in[n]}$, $B=\set{b_j:j\in [n]}$ and an edge $\set{a_i,b_j}$ for every edge $(i,j)$ of $D$. If $D$ satisfies $\min\set{\d^-,\d^+}\geq 1$ then $G(D)$ has minimum degree at least one and vice-versa. Let $G_\bz=G(D_\bz)$ for $\bz\in \Omega_1$.

Our input will be a sequence \by\ constructed as in Claim \ref{cl1}, starting with a random $\bx\in \Omega_1$. We will discuss Hamiltonicity in the context of $D_\by$. This  phase  partitions $E_\by$ into three subsets $E_1,E_2,E_3$ each with different purpose. We first show that the digraph induced by $E_1$ contains a collection $\c{C}_1$ of $O(\log n)$ vertex disjoint directed cycles that cover all vertices, i.e. a vertex cycle cover. We then use $E_2$ to transform $\c{C}_1$ into a vertex cycle cover $\c{C}_2$ in which each cycle $C\in \c{C}_2$ has size at least $n/\log^{1/2}n$. We use $E_3$ to transform $\cC_2$ into a Hamilton cycle.

Let $SMALL$ denote the set of vertices $v$ such that $d^+_\bx(v)+d^-_\bx(v)\leq \log n/100$.  Let $ALTERED$ denote the set of vertices that are involved in switches. Let $dist_\bz(v,w)$ denote the distance between vertices $v$ and $w$ in either $D_\bz$ or $G_\bz$, depending on context. Paths in $G_\bz$ when considered in $D_\bz$ alternate in orientation.
\begin{lemma}\label{far}
Suppose that $m\geq \frac25n\log n$. Suppose that \by\ is the result of switchings w.r.t. a random $\bx\in \Omega_1$. Then w.h.p.,
\begin{enumerate}[(a)]
\item $|SMALL|\leq n^{7/8}$.
\item No vertex is involved in four or more switches.
\item No member of $SMALL$ is incident to an edge involved in any of the switches used to construct \by.
\item No member of $SMALL$ lies on a cycle of $D_\by$ or of $G_\by$ of length at most $\ell_0=\log n/(20\log\log n)$.
\item $u,v,w\in SMALL\cup ALTERED$ implies that $dist_\by(u,v)+dist_\by(v,w)\geq \ell_0$ in $D_\by$ or in $G_\by$.
\end{enumerate}
\end{lemma}
\begin{proof}
(a)
\begin{align}
\E(|SMALL|)&\leq 2n\Pr\brac{d_\bx^+(1)\leq \log n/100\
\mid \cE'}\nn\\
&=2n\sum_{k=1}^{\log n/100}\Pr\brac{(d_\bx(1)=k)\wedge\brac{\sum_{i=2}^nZ_i'=m-k} \bigg|\  \cE'}\nn\\
&=2n\sum_{k=1}^{\log n/100}\frac{\Pr(Z_1'=k)\Pr\brac{\sum_{i=2}^nZ_i'=m-k}}{\Pr\brac{\sum_{i=1}^nZ_i'=m}}\label{<<}\\
&\sim 2n\sum_{k=1}^{\log n/100}\Pr(Z_1'=k)\label{<<<}\\
&=2n\sum_{k=1}^{\log n/100}\frac{z^k}{k!(e^z-1)}\nn\\
&\leq 3n\bfrac{ze}{\log n/100}^{\log n/100}e^{-z}\nn\\
&\leq n^{3/4}.\nn
\end{align}
We use Proposition \ref{propx} from Appendix B to obtain \eqref{<<<} from \eqref{<<}. The final inequality uses the fact that $z\sim m/n$ if $m/n\to\infty$. Our bound on $|SMALL|$ follows from the Markov inequality.

(b) The probability that a vertex is part of the randomly chosen edge of an $L$-switch is at most $\D_0/m$. Here, throughout the paper,
\[
\D_0=\log^2n
\]
 is a high probability bound on the maximum degree, (proved in Lemma \ref{LM} below). So the probability there is some vertex that is chosen twice in this way is at most $n(\D_0/m)^2=o(1)$. Apart from this, a vertex can be involved twice, once in a $P$-switch and once in an $L$-switch, w.h.p. See Lemma \ref{LM} in the appendix for the constraints on loops and parallel edges.

(c) Assuming (a), the probability that (c) fails is w.h.p. at most $O(n^{7/8}\log^4n\cdot n^{-1})$. The factor $\log^4n$ comes from the bound of $O(\log^4n)$ on the number of switches needed, which in turn follows from Lemma \ref{LM}(b)(g) below.

(d) The probability that (d) fails can be bounded by
\mults{
2\sum_{k=2}^{\ell_0}k\binom{n}{k}k!\binom{m}{k}\bfrac{\D_0}{m}^{2k}\sum_{\ell=0}^{\log n/100} \binom{m-k}{\ell}\frac{1}{n^\ell}\brac{1-\frac{1}{n}}^{m-k-\ell}\\
\leq 3(\log n)^{2\log n/100}(100e)^{\ell_0}n^{-1/3+o(1)}=o(1).
}
{\bf Explanation:} After choosing the vertices of the cycle and then selecting a vertex of the cycle in at most $\binom{n}{k}k!k$ ways, we bound the probability that the edges exist by $\binom{m}{k}\bfrac{\D_0}{m}^{2k}$ and then bound the probability that a selected vertex has few neighbours outside the cycle.

(e) The probability that (e) fails can be bounded by
\mults{
\sum_{k=1}^{\ell_0}\binom{n}{k+1}k!\binom{k}3\binom{m}{k}\bfrac{\D_0}{n}^{2k}\brac{\bfrac{3e^2\log^4n\D_0}{m}+\sum_{\ell=0}^{\log n/100}\binom{m-k}{\ell}\frac{1}{n^\ell}\brac{1-\frac{1}{n}}^{m-k-\ell}}^3\\
=o(1).
}
{\bf Explanation:} this is similar to (d). the expression $\binom{n}{k+1}k!k$ counts paths $u\to v\to w$ and $\binom{m}{k}\bfrac{\D_0}{n}^{2k}$ bounds the probability the edges exist. The large bracket is the sum of the probability of being part of a switch and the probability that a selected vertices have few neighbours outside the path.
\end{proof}
We now discuss the construction of $E_1,E_2,E_3$. Starting with a random $\bx\in \Omega_1$ we construct $\by=y_1,y_2,\ldots,y_{2m}$ randomly as in Claim \ref{cl1}. We let $j_1=n\log n/5$ and initialise $E_1$ as $\set{(y_{2j-1},y_{2j}):j\leq j_1}$. Then we go through \by\ in order $j=n\log n/5+1,\ldots,m$ and add $(y_{2j-1},y_{2j})$ to $E_1$ if either $y_{2j-1}$ has appeared less than $\log n/100$ times in $y_1,y_3,\ldots,y_{2j_1-1}$ or if $y_{2j}$ has appeared less than $\log n/100$ times in $y_2,y_4,\ldots,y_{2j_1}$. We randomly place the edges $E_\by\setminus E_1$ into $E_2,E_3$, each with probability 1/2, and then re-order the edges so that $E_2\cap M$ appears before $E_3\cap M$ where $M$ is the set of edges induced by $y_i,i> n\log n/5$. This re-ordering is not very important, it just helps in our description of the constructions. We just need two large sets of sufficiently random edges.

\subsection{Phase 1}
Let $G_i$ be the bipartite graph with bipartition $A,B$ that is induced by the edges $E_i,i=1,2,3$ and let $F_i=E(G_i)=\set{\{a_r,b_s\}:(r,s)\in E(G_i)}$. For $S\subseteq A$, we define $N_1(S)$ to the set of $G_1$-neighbours of $S$ in $B$
\begin{lemma}\label{pm}
Let $m=\tfrac n2(\log n+2\log\log n+c_n)$ satisfy $m/n\lesssim\log n$, where $A_n\lesssim B_n$ if $A_n\leq (1+o(1)B_n$ as $n\to\infty$. (This bounds $c_n$ and for larger $c_n$ we have minimum in- and out-degree at least one w.h.p.) Then
\beq{M1}{
\lim_{n\to\infty}\Pr(G_1\text{ has a perfect matching })=\begin{cases}0&c_n\to-\infty.\\e^{-e^{-c}/8}&c_n\to c.\\1&c_n\to\infty.\end{cases}
}
\end{lemma}
\begin{proof}
Let $\cM_1$ denote the event that $G_1$ contains two vertices of degree one in either $A$ or $B$ that share a common neighbour. Let $\cN_k,k\geq 3$ denote the event that $G_1$ contains a set $S\subseteq A:|S|=k$  such that $|N_1(S)|\leq k-1$. A minimal such $S$ satisfies (i) $|N_1(S)|=k-1$ and (ii) each $b\in N_1(S)$ has at least two neighbours in $S$, (we can reduce $S$ by one if $|N_1(S)|<k-1$ or if there is a $b\in N_1(S)$ with one neighbour in $S$) and (iii) $S\cup N_1(S)$ induces a connected subgraph. If $\cM_1$ does not occur then this reduction in the size of $S$ must stop before $k=1$ because of our minimum degree condition. Then, because $G_1$ has no isolated vertices,
\beq{M2}{
\Pr(\cM_1)\leq \Pr(G_1\text{ does not have a perfect matching})\leq \Pr(\cM_1)+2\sum_{k=3}^{n/2}\Pr(\cN_k).
}
As usual, we deal with $k>n/2$ by looking at the neighbourhood of sets $T\subseteq B$ of size $n-k+1$ and using symmetry.

We first deal with $\Pr(\cM_1)$.

{\bf Case 1: $c_n=-\om$ where $\om\to \infty$ and $m/n\to\infty$:}\\
Let an {\em A-cherry} be a path of length two with centre vertex in $A$ and endpoints of degree one in $B$. We define B-cherries by reversing the roles of A and B. We let $Z$ denote the number of A-cherries. Then in $D_\bx$, with $\cE',\cE''$ as in \eqref{Edash}, $\cD_d=\set{d_\bx^-(1)=d_\bx^-(2)=1,d_\bx^+(3)=d}$,
\begin{align}
\E(Z)&=n\binom{n}{2}\sum_{\cD_d\geq 2}\Pr(\cD_d\mid\cE',\cE'')\frac{m(m-1)d(d-1)}{m^2(m-1)^2}\nn\\
&=n\binom{n}{2}\sum_{d\geq 2}\frac{\Pr\brac{\cD_d,\sum_{i\geq 3}Z_i''=m-2,\sum_{i\geq 2}Z_i'=m-d}}{\Pr\brac{\sum_{i}Z_i''=m,\sum_{i}Z_i'=m}}\frac{d(d-1)}{m(m-1)}\nn\\
&=n\binom{n}{2}\sum_{d\geq 2}\frac{\Pr(\cD_d)\Pr\brac{\sum_{i\geq 3}Z_i''=m-2}\Pr\brac{\sum_{i\geq 2}Z_i'=m-d}}{\Pr\brac{\sum_{i}Z_i''=m}\Pr\brac{\sum_{i}Z_i'=m}}\frac{d(d-1)}{m(m-1)}\nn\\
&\sim n\binom{n}{2}\sum_{d\geq 2}\Pr(\cD_d)\cdot\frac{ d(d-1)}{m(m-1)}\label{EZ}\\
\noalign{\text{after removing  the probabilities of  sums of $Z_i',Z_i''$ using Proposition \ref{propx} in Appendix B}}
&\sim \frac{n^3}{2m^2}\brac{\sum_{d=2}^{\D_0}\frac{z^{d+2}}{(d-2)!(e^z-1)^3}+\Pr(d_\bx^+(3)\geq \D_0)}, \label{EZ}\\
&\sim \frac{n^3z^4}{2m^2e^{2z}},\quad\text{ since $\Pr(d_\bx^+(3)\geq \D_0)=(\log n)^{-\Omega(\log^2n)}$ and $e^z-1\sim e^z$ as $m/n\to\infty$},\nn\\
&\sim \frac{m^2}{2ne^{2m/n}},\quad\text{ since $m/n\to\infty$}\label{EZ1}\\
&\geq e^{\om}/8\to\infty.\label{EZ2}
\end{align}
It is worth observing that the above shows that conditioning on $\cE',\cE''$ only changes probabilities by a factor $1+o(1)$ and from now on we rely on this without mention.
 
{\bf Case 1a: $\om\leq \a\log n$ where $\a=2/7$:}\\
At this point we want to use the Chebyshev inequality to prove that $Z\neq 0$ w.h.p. Thus we write $\E(Z(Z-1))=A_1+A_2$, where $A_1$ counts disjoint cherries and $A_2$ counts pairs of A-cherries sharing an A-vertex (forming a $K_{1,3}$). In this case, $z\geq (5-o(1))\log n/14$, which is needed for \eqref{nn1} below.
\begin{align}
&A_1=\nn\\
&n(n-1)\binom{n}{2}^2\sum_{d_1,d_2\geq 2}\Pr(d_\bx^-(1)=d_\bx^-(2)=d_\bx^-(3)=d_\bx^-(4)=1,d_\bx^+(5)=d_1,d_\bx^+(6)=d_2\mid\cE',\cE'')\times\nn\\
&\hspace{3in}\frac{m(m-1)d_1(d_1-1)}{m^2(m-1)^2}\frac{(m-2)(m-3)d_2(d_2-1)}{(m-2)^2(m-3)^2}\nn\\
&\hspace{4.5in}\sim \E(Z)^2,\qquad\text{ using \eqref{EZ}}.\nn\\
&A_2=\nn\\
&n\binom{n}{3}\sum_{d\geq 3}\Pr(d_\bx^-(1)=d_\bx^-(2)=d_\bx^-(3)=1,d_\bx^+(4)=d\mid\cE',\cE'') \frac{m(m-1)(m-2)d(d-1)(d-2)}{m^2(m-1)^2(m-2)^2}\nn\\
&\sim \frac{n^4}{6m^3}\sum_{d\geq 3}\frac{z^{d+4}}{(d-3)!(e^z-1)^4}\sim \frac{m}{6e^{3z}}=O(n^{o(1)+(3\a-1)/2})=o(1).\label{nn1}
\end{align}
Thus $\E(Z(Z-1))\sim \E(Z)^2$ and so from \eqref{EZ2}, $Z\neq 0$ w.h.p.

{\bf Case 1b: $\om\geq \a\log n$ and $m\geq n$:}\\
Let $\wh \Omega_1$ be the set of sequences obtained by generating $Z_i',Z_i'',i\in[n]$ and then randomly permuting the multi-sets $\set{j\times Z_i'},\set{j\times Z_i''}$ and placing the obtained sequences in the odd and even positions respectively. If $\sum_{i=1}^nZ_i'>\sum_{i=1}^nZ_i''$ then use *'s to fill in any blanks and vice-versa. Arguing as for \eqref{EZ2} we see that $\E(Z)=\Omega(n^{2\a})$. Let $\wh Z$ be the number of A-cherries $Z$ computed w.r.t. $\wh\Omega_1$, when the $Z_i',Z_i''$ are replaced by $\wh Z_i'=\min\set{Z_i',\D_0},\wh Z_i''=\min\set{Z_i'',\D_0}$ respectively. now changing the value of a $\wh Z_i',\wh Z_i''$ can only change $\wh Z$ by at most $\D_0^2$. It follows from McDiarmid's martingale inequality that for any $t>0$,
\[
\Pr(|\wh Z-\E(\wh Z)|\geq t)\leq \exp\set{-\frac{2t^2}{2n\D_0^4}}.
\]
Putting $t=\frac12\E(Z)$ implies that $\Pr(Z=0)=O(e^{-n^{4\a-1}})$. We choose $\a=2/7$ to give $\Pr(Z=0)=O(e^{-n^{1/7}})$. We need to add in the probability that $\wh Z\neq Z$, but this can be bounded by $O(n(\log n)^{-\log^2n})$. This completes Case 1.

{\bf Case 2: $c_n\to c$:}\\
We first rule out A-cherries that share a vertex i.e. three vertices of degree one having a common neighbour. Let this event be $\cE_3$. We can because of Lemma \ref{far} consider this event in $D_\bx$ i.e. in the digraph before we do any switches. This is because vertices in SMALL do not participate in switches. Then
\begin{align}
\Pr(\cE_3)&\leq 2n\binom{n}{3}\Pr(N_1(\set{a_1,a_2,a_3})=\set{b_1}\mid \cE_1',\cE_1'')\nn\\
&\leq(1+o(1))\sum_{d=3}^{\D_0}n^4d^3m^{-3}\brac{1-\frac{3}{n}}^{m-3}+(\log n)^{-\Omega(\log^2n)}.\label{M3}
\end{align}
{\bf Explanation:} in the above expression, $d$ is the degree of vertex $b_1$ and Lemma \ref{LM} below shows that $d\leq \D_0$ with probability $1-(\log n)^{-\Omega(\log^2n)}$. The 2 accounts for the common neighbour being in $A$ or $B$. Having chosen $a_1,a_2,a_3,b_1$ in at most $n\binom{n}{3}$ ways, we choose the relevant edges in at most $m^3$ ways and then $m^3d^3m^{-6}$ approximates the probability that the edges are as claimed.

It follows from \eqref{M3} that $\Pr(\cE_3)=O(n^{-1/2+o(1)})=o(1)$.

We are going to apply the method of moments to estimate the distribution of $Z$. Fix $k \ge 1$. Then,
\[
\E((Z)_k)\sim  n^{3k}\sum_{t=0}^k\binom{k}{t}m^{2k}n^{-4k}\brac{1-\frac{2k}{n}}^{m-2k} \sim \bfrac{e^{-c}}{8}^k.
\]
It follows from the method of moments that $Z$ is asymptotically Poisson with mean $e^{-c}/8$.

{\bf Case 3: $c_n\to\infty$, $m\leq (1+o(1))n\log n$.:}\\
Going back to \eqref{EZ2}, with $c_n=\om$, we see that we see that $\E(Z)\leq e^{-\om}/8\to 0$.

We have dealt with $\cM_1$ and we now tackle the sum in \eqref{M2}. We can assume that $c_n\geq c$ for some constant $c$.

We will be more careful with our bound on maximum degree. A first moment calculation bounds $\max_id_\bx(i)$ by $5\log n$. Then, together with Lemma \ref{far}(b) this allows us to bound the maximum degree in \by\ by $\D_1=6\log n$. We remind the reader that $\ell_0=\log n/20\log\log n$.

{\bf Case 1: $3\leq k\leq \ell_0/2$.}
Suppose then that $S\subseteq A$ and $3\leq |S|\leq \ell_0/2$. It follows from Lemma \ref{far}(e) and the fact that $S\cup N_1(S)$ forms a connected set that $|S\cap (SMALL\cup ALTERED)|\leq 2$. By construction, if $v\notin SMALL$, then $d_1(v)\geq \log n/100$, where $d_i$ denotes degree in $G_i$. Thus $|N_1(S)|>\ell_0/2$.

{\bf Case 2: $\ell_0/2<k\leq k_0=\frac{m}{10\D_1}$.}
Suppose then that $S\subseteq A$ and $\ell_0/2<|S|\leq k_0$. Very crudely, it follows from Lemma \ref{far} and the fact that $S\cup N_1(S)$ forms a connected set that $|S\cap (SMALL\cup ALTERED)|\leq |S|/2$. Thus,
\[
\Pr(\exists S)\leq \sum_{k=\ell_0/2}^{k_0}\binom{n}{k}^2\bfrac{k\D_1}{m}^{k\log n/200}\leq \sum_{k=\ell_0/2}^{k_0}\brac{\bfrac{ke}{m}^{\log n/200-2}\cdot\D_1^{\log n/200}}^k=o(1).
\]
{\bf Case 3: $k_0<k\leq n/2$.}
Suppose then that $S\subseteq A,T\subseteq B$ and $k_0<|S|=|T|\leq n/2$ and $N_1(S)\subseteq T$. By bounding the probability that $S$ has no neighbours outside $T$, we have
\mults{
\Pr(\exists S)\leq \sum_{k=k_0}^{n/2}\binom{n}{k}^2\brac{1-\frac{(k-|SMALL\cup ALTERED|)\log n}{100m}}^{k\log n/200}\leq\\ \sum_{k=k_0}^{n/2}\brac{\bfrac{ne}{k}^2\cdot\exp\set{-\frac{k\log^2n}{20000m}}}^k=o(1),
}
since $k_0=\Omega(n)$ here.
\end{proof}

It is easy to show by symmetry that the perfect matching $M_1$ promised by the above can be taken to be a uniform random perfect matching. Indeed, given $D$ and a permutation $\p$ of  $[n]$ we define $\p D=([n],\set{(i,\p(j)):(i,j)\in E(D)}$. Let $\p$ be a uniform random permutation of $B$. Suppose that we check for a perfect matching in $D$ by looking for a perfect matching $M$ in $\p D$ and then taking $\p^{-1}M$ to be our perfect matching $M_1$ in $D$. Note that $\p^{-1}M$ is uniformly random, no matter what value $M$ takes. Finally note that $\p D$ has the same distribution as $D$ so that $\Pr(G_D\text{ has a perfect matching})=\Pr(G_{\p D}\text{ has a perfect matching})$.
\subsection{Phase 2}
We can from now on assume that $c_n$ does not tend to $-\infty$, using the notation of Lemma \ref{pm}. There is nothing more to prove if $c_n\to-\infty$.  We can also assume that $m\leq (1+o(1))n\log n$, for otherwise we can refer to \cite{F1}. So now let $m=an\log n$ where $a\in [1/2-o(1),1+o(1)]$, and that $M_1$ exists.

Let $\c{C}_1$ be the cycle cover in $D$ corresponding to $M_1$. $M_1$ being a uniform random perfect matching implies that w.h.p. $\c{C}_1$ has at most $2\log n$ cycles. In this phase we transform $\c{C}_1$ into a cycle cover in which each cycle has size at least $n/\log^{1/2}n$. 
We do this using the edges $E_2$. See Section \ref{2.4} for a description of how this is done.

Let $Y_1$ be the set of indices $i$ such that $i$ appears at most $\log n/100$ times in the sequence $y_1,y_2,\ldots,y_{n\log n/5}$.
\begin{lemma}\label{smallY1}
$|Y_1|\leq n^{0.999}$ w.h.p.
\end{lemma}
\begin{proof}
Let $W_i'$ be the number of times $i$ appears in $x_1,x_3,\ldots,x_{n\log n/5}$ where \bx\ is chosen randomly from $\wh\Omega_1$, where $\wh\Omega_1$ is defined in Case 1b, following \eqref{nn1} above. Let $X_1'=\set{i:W_i'\leq \log n/100}$. Then, with $\cE'$ as in \eqref{Edash} and where $m_1=n\log n/5$ and $L=\log n/100$,
\begin{align*}
\E(|X_1'|)&=n\sum_{k\geq 1}\Pr(Po(z)=k\mid \cE')\sum_{r=0}^L\frac{\binom{m_1}{r}\binom{m-m_1}{k-r}}{\binom{m}{k}}
\\
&\leq n\frac{\Pr(Po(z)\geq 10\log n)}{\Pr(\cE')}+n\sum_{k=1}^{10\log n}\Pr(Po(z)=k\mid \cE')\sum_{r=0}^L\frac{\binom{m_1}{r}\binom{m-m_1}{k-r}}{\binom{m}{k}}\\
&\leq o(n^{-1})+(1+o(1))n\sum_{k=1}^{10\log n}\Pr(Po(z)=k)\sum_{r=0}^L\frac{\binom{m_1}{r}\binom{m-m_1}{k-r}}{\binom{m}{k}}\\
&=o(n^{-1})+(1+o(1))n\sum_{k=1}^{10\log n}\Pr(Po(z)=k)\sum_{r=0}^L\binom{k}{r}\bfrac{m_1}{m}^r\brac{1-\frac{m_1}{m}}^{k-r}\\
&=o(n^{-1})+(1+o(1))n\sum_{k=1}^{10\log n}\frac{z^k}{(e^z-1)k!}\sum_{r=0}^L\binom{k}{r}\bfrac{2}{5}^r\bfrac{3}{5}^{k-r}\\
&=o(n^{-1})+n^{1/2+o(1)}\sum_{k=1}^{10\log n}\frac{z^k}{k!}\sum_{r=0}^L\binom{k}{r}\bfrac{2}{5}^r\bfrac{3}{5}^{k-r}.
\end{align*}
Now, $z\sim  a\log n,a\in [1/2-o(1),1+o(1)]$ (see \eqref{zrho}) and so
\beq{smallk}{
\sum_{k=1}^{a\log n/2}\frac{z^k}{k!}\leq 2\bfrac{2e^{1+o(1)}z}{a\log n}^{a\log n/2}\leq n^{9/10}.
}
Furthermore, if $k\geq a\log n/2$ then the Chernoff bounds imply that
\[
\sum_{r=0}^L\binom{k}{r}\bfrac{2}{5}^r\bfrac{3}{5}^{k-r}\leq \exp\set{-\frac12\bfrac{49}{50}^2\frac{2k}{5}}\leq e^{-k/11}.
\]
Therefore
\beq{largek}{
\sum_{k=a\log n/2}^{10\log n}\frac{z^k}{k!}\sum_{r=0}^L\binom{k}{r}\bfrac{2}{5}^r\bfrac{3}{5}^{k-r}\leq \sum_{k\geq a\log n/2}\frac{z^k}{k!}e^{-k/11}\leq \exp\set{ze^{-1/11}}\leq n^{0.46}.
}
The Markov inequality and \eqref{smallk} and \eqref{largek} imply that $|X_1'|\leq n^{0.99}$ w.h.p.

To finish the proof we use $|Y_1|\leq|X_1|+|ALTERED|$, where $X_1=X_1'\cup X_1''$ and $X_2''$ is defined w.r.t. even indices. We note that $|ALTERED|=O(\log^4n)$ which follows from Lemma \ref{LM}(b)(g) below.
\end{proof}
It follows that w.h.p.
\[
|E_2\cup E_3|\geq (1+o(1))n\log n/2-n\log n/5-n^{0.999}\log n/100\geq n\log n/3.
\]
Furthermore, if we fix the degrees of the digraph induced by $E_2\cup E_3$ then the actual edges are randomly distributed. We can construct $E_2\cup E_3$ by only checking that vertices $x_{2j-1},x_{2j},j\leq n\log n/5$ have appeared $\log n/100$ times before, without actually identifying them.

Let  $Y_2(k)$ be the set of indices $i\notin Y_1$ such that $i$ occurs at most $k+\log n/100$ times as an odd index in the sequence $y_{n\log n/5+1},\ldots,y_{n\log n/3}$ or at most $k+\log n/100$ times as an even index. Note that we have ordered $E_2,E_3$ so that $E_3$ follows $E_2$ after index $n\log n/5$.
\begin{lemma}\label{lcore1}
Let $L=\log\log n$. Then

(a) $|Y_2(L)|\leq n^{15/16}$ w.h.p.

(b) W.h.p. every connected set of vertices in $G_i,i=2,3$ of size at most $L$ contains at most 20 members of $Y_2(L)$.

(c) No member of $Y_2(L)$ lies on a cycle in $G_\by$ of length at most $L$. (Here $G_\by$ is the graph induced by the sequence \by.)
\end{lemma}
\begin{proof}
(a)  We have
\mults{
\E(|Y_2(L)|)\leq 2n\Pr(Bin(|E_1\cup E_2|,1/n)<L+\log n/100 )\\
\leq 2n\binom{2n\log n/15}{L+\log n/100}n^{-L-\log n/100}\brac{1-\frac{1}{n}}^{2n\log n/15-L-\log n/100}\leq n^{14/15}.
}
So the Markov inequality implies that $|Y_2(L)|\leq n^{15/16}$ w.h.p.

(b) The probability that there exists such a set can be bounded by
\mults{
2\sum_{\ell=20}^{L}\binom{n}{\ell}\binom{2n\log n/15}{\ell-1}\ell^{\ell-2}2^\ell(\ell-1)! n^{-2\ell+2}\times\\
\brac{\sum_{i=0}^{L+\log n/100}\binom{2n\log n/15}{i}n^{-i}\brac{1-\frac{1}{n}}^{2n\log n/15-i}}^\ell
\leq 2\sum_{\ell=20}^Ln^{2+o(1)}\times n^{-\ell/8}=o(1).
}
{\bf Explanation:} $\binom{n}{\ell}$ counts the number of choices for the vertices in our connected set $S$ of size $\ell$. $\binom{2n\log n/15}{\ell-1}$ counts the number of choices for the positions of the $\ell-1$ edges of a tree $T$ contained in $S$ in the ordering of $E_i$. There are $\ell^{\ell-2}$ choices for $T$, $2^\ell(\ell-1)!$ choices for the ordering of the edges in the sequence \by\ and probability $n^{-2\ell+2}$ that the edges are the ones defining $T$. The large bracketed term bounds the probability that the $\ell$ chosen vertices are in $Y_2(L)$.

(c) The probability that we can find such a cycle can be bounded by
\mults{
\sum_{\ell=3}^L\ell\ell!\binom{n}{\ell}\binom{m}{\ell}\bfrac{\D_1}{m}^{2\ell}\times\\
\brac{\Pr([\ell]\cap ALTERED\neq \emptyset)+\sum_{i=1}^{L+\log n/100}\Pr\brac{d_\bx(1)=i\ \bigg|\  \sum_{i=1}^nZ_i'=\sum_{i=1}^nZ_i''}}=o(1).
}
{\bf Explanation:} we choose the vertices of the cycle $C$ in $\binom{n}{\ell}$ ways and choose a vertex of $C$ in $\ell$ ways and an ordering of the vertices of $C$ in $\ell!$ ways. Then we choose the places of the edges of the cycle in \bx\ in $\binom{m}{\ell}$ ways. Then $\bfrac{\D_1}{m}^\ell$ bounds the probability that these places in the sequence contain the claimed edges. recall that $\D_1=6\log n$ is a bound on the maximum degree in \bx. If $C\cap ALTERED=\emptyset$ then we can replace \bx\ by \by.
\end{proof}
We now show that the $\log n/100$ cores $C_2,C_3$ of $G_2,G_3$ are large.
\begin{lemma}\label{ars}
$C_2,C_3\supseteq ([n]\setminus Y_2(L))$ w.h.p
\end{lemma}
\begin{proof}
Let $S_i,i\geq 0$ denote the set of vertices stripped off in the $i$th round of the core finding process. Here $S_0=Y_2(0)$ and we argue by induction that $v\in S_i$ implies that $v\in Y_2(L)$ and that $v$ has at most 20 neighbours in $S_{i-1}$. This is true for $i=1$ by Lemma \ref{lcore1}(b). This is also true for $i>1$ for the same reason, given the inductive assumption. But then if $S(21)\neq \emptyset$, there is a path of length 21 consisting of vertices in $Y_2(L)$. So, $S(i)\subseteq Y_2(L)$ for $i\geq 0$.
\end{proof}
Thus w.h.p. $E_i,i=2,3$ contains a set $K_i$ of $n-n^{0.999}$ vertices that induces a subgraph with minimum degree at least $\log n/100$. We now show how to use $E(K_2)$ to increase the minimum cycle size of our perfect matching to at least $n_0=n/\log^{1/2}n$. Note that $n^{0.999}=o(n_0)$. Thus, the vertices not in $K_i$ can only represent a small part of each large cycle. 

\subsection{Elimination of Small Cycles}\label{2.4}
We partition the cycles $\c{C}_1$ associated with $M_1$ into {\em small} cycles $C,|C| < n_0$ and {\em large} $|C| \geq n_0$ respectively. We define a {\em Near Permutation Digraph} (NPD) to be a digraph obtained from a {\em Permutation Digraph} (PD) (cycle cover) by removing one edge. Thus an NPD $\Gamma$ consists of a path $P(\Gamma)$ plus a permutation digraph $PD(\Gamma)$ which covers $[n]\setminus V(P(\Gamma))$.

In a random permutation the expected number of vertices on cycles of length at most $s$ is precisely $s$ (\cite{K}). Thus, by the Markov inequality, w.h.p. $\Gamma_0$ contains at most $n\log\log n/\log^{1/2}n$ vertices on small cycles. Condition on this event and refer to it as $\c{M}$.

We now give an informal description of a process which removes a small cycle $C$ from a {\em current} PD $\Pi$. We start by choosing an (arbitrary) edge $(v_0,u_0)$ of $C$ and then delete it to obtain an NPD $\Gamma_0$ with $P_0=P(\Gamma_0)\in {\cal P}(u_0,v_0)$, where ${\cal P}(x,y)$ denotes the set of paths from $x$ to $y$ in $D$. The aim of the process is to produce a {\em large} set $S$ of NPD's such that for each $\Gamma\in S$, (i) $P(\Gamma)$ has at least $n_0$ edges and (ii) the small cycles of $PD(\Gamma)$ are a subset of the small cycles of $\Pi$. We will show that w.h.p. the endpoints of one of the $P(\Gamma)$'s can be joined by an edge to create a permutation digraph with (at least) one less small cycle.

The Out-Phase consists of a sequence of {\em basic steps.} In a  basic step of the {\em Out-Phase} we have an NPD $\Gamma$ where $P(\Gamma)$ is a path from $u_0$ to another vertex $v$. We examine the edges of the core $C_2$ leaving $v$ i.e. the edges going {\em out} from the end of the path. Let $w$ be the terminal vertex of such an edge and assume that $\Gamma$ contains an edge $(x,w)$. Then $\Gamma'=\Gamma\cup\{(v,w)\}\setminus\{(x,w)\}$ is also an NPD. $\Gamma'$ is acceptable if (i) $P(\Gamma')$ contains at least $n_0$ edges and (ii) any new cycle created (i.e. in $\Gamma'$ and not $\Gamma$) also has at least $n_0$ edges.

If $\Gamma$ contains no edge $(x,w)$ then $w=u_0$. We accept the edge if $P$ has at least $n_0$ edges. This would (prematurely) end an iteration, although it is unlikely to occur.

We do not want to look at very many edges of $E(K_2)$ in this construction and we build a tree $T_0$ of NPD's in a natural breadth-first fashion where each non-leaf vertex $\Gamma$ gives rise to NPD children $\Gamma'$ as described above. The construction of $T_0$ ends when we first have $\nu=n^{1/2}\log n$ leaves. The construction of $T_0$ constitutes an Out-Phase of our procedure to eliminate small cycles. Having constructed $T_0$ we need to do a further {\em In-Phase}, which is similar to a set of Out-Phases.

Then w.h.p. we close at least one of the paths $P(\Gamma)$ to a cycle of length at least $n_0$. If $|C|\geq 4$ and this process fails then we try again with a different independent edge of $C$ in place of $(u_0,v_0)$.

We now increase the formality of our description. We start Phase 2 with a PD $\Pi_0$, say, and a general iteration of Phase 2 starts with a PD $\Pi$ whose small cycles are a subset of those in $\Pi_0$. Iterations continue until there are no more small cycles. At the start of an iteration we choose some small cycle $C$ of $\Pi$. There then follows an Out-Phase in which we construct a tree $T_0=T_0(\Pi,C)$ of NPD's as follows: the root of $T_0$ is $\Gamma_0$ which is obtained by deleting an edge $(v_0,u_0)$ of $C$.

We grow $T_0$ to a depth at most $i_0=\rdup{1.5\log n}$. The set of nodes at depth $t$ is denoted by $S_t$.
\\
Let $\Gamma \in S_t$ and $P=P(\Gamma)\in {\cal P}(u_0,v)$. The {\em potential}
children $\Gamma'$ of $\Gamma$, at depth $t+1$ are defined as follows.

Let $w$ be the terminal vertex of an edge directed from $v$ in $E(K_2)$.
\\
{\bf Case 1.} $w$ is a vertex of a cycle $C' \in PD(\Gamma)$ with edge $(x,w) \in C'$.\\
Let $\Gamma'=\Gamma\cup\{(v,w)\}\setminus \{(x,w)\}$, thus extending the path $P$.
\\
{\bf Case 2}. $w$ is a vertex of $P(\Gamma)$.\\
 Either $w = u_0$, or $(x,w)$ is an edge of $P$. In the former case $\Gamma\cup\{(v,w)\}$ is a PD $\Pi'$ and in the latter case we let $\Gamma'=\Gamma\cup\{(v,w)\}\setminus \{(x,w)\}$, making a cycle and a shorter path.

In fact we only admit to $S_{t+1}$ those $\Gamma'$ which satisfy the following conditions {\bf C}(i) and {\bf C}(ii).

{\bf C}(i) The new cycle formed (Case 2 only) must have
at least $n_0$ vertices, and the path formed (both cases) must either
be empty or have  at least $n_0=n/\log^{1/2}n$ vertices. When the path formed is empty we close the iteration and if necessary start the next with $\Pi'$.

We denote the the $E_i$-out-neighbors of $v\in K_2$ by $out(v)$. We define a set $W$ of {\em used} vertices. Initially all vertices in $K_2$ are {\em unused} i.e. $W=\emptyset$. Whenever we examine an
edge $(v,w)$, we add both $v$ and $w$ to $W$. So if $v\not\in W$ then it has at least $\log n/100$ random choices $out(v)$ of neighbours to make in $K_2$. Furthermore these choices are almost uniform in that no vertex has a more than $600/n$ chance of being chosen. This is because the in-degrees in $C_2$ vary between $\log n/100$ and $\D_1=6\log n$. We do not allow $|W|$ to exceed $n^{3/4}$.

{\bf C} (ii) $x,w \not \in W$ .

An edge $(v,w)$ which satisfies the above conditions is described as {\em acceptable}.

\begin{lemma}\label{lem9}
Let $C$ be small, i.e. $|C|\leq n_0=n/\log^{1/2}n$. Let $t_0\sim \frac{\log n}{2\log\log n}$ be the smallest integer such that $(\log n/100)^t\geq \n=n^{1/2}\log n$ then
$$\Pr(\exists t\leq t_0\text{ such that }|S_{t}|  \in[\nu,\D_1\nu] ) =1-o(n^{-2}).$$
\end{lemma}
{\em Proof.} We assume we stop an iteration, in mid-phase if necessary, when $|S_t| \in[\nu,3\nu]$. Let us consider a generic construction
in the growth of $T_0$. Thus suppose we are extending from $\Gamma$ and $P(\Gamma)\in{\cal P}(u_0,v)$.

We consider $S_{t+1}$ to be constructed in the following manner: we first examine $out(v), v\in S_t$ in the order that these vertices were placed in $S_t$ to see if they produce acceptable edges. We then add in those vertices $x\not\in W$ which arise from $(x,w)$ with $v=in(w)\in S_t,w\not\in W$. (Here $in(w)$ is defined analogously to $out(v)$.)

Let $Z(v)$ be the number of acceptable vertices in $out(v)$. If $w\in out(v)$ is unacceptable then either (i) $w$ lies on $P(\Gamma)$ and is too close to an endpoint; this has probability bounded above by $600/\log^{1/2}n$, or (ii) the corresponding vertex $x$ is in $W$; this has probability bounded above by $600n^{-1/4}$, or (iii) $w$ lies on a small cycle. Thus $Z(v)$ stochastically dominates $B(\log n/100,p)$ where $p= 1-\frac{600\log\log n}{\log^{1/2}n}$, regardless of the history of the process. Here we use the fact that we have conditioned on $\c{M}$ at the start.

Now $S_0=\set{v_0}$ and suppose that $S_t=\set{v_0,v_1,\ldots,v_k}$ and that we expose $out(v_i)$ in the order $1,2,\ldots,k$ and {\em update} $W$ as we go. Then as long as $|W|\leq n^{3/4}$, we find that $|S_{t+1}|$ stochastically dominates $Bin(|S_t|\log n/100,p)$. Thus first considering $S_1$, we have
\[
\Pr\brac{|S_1|\leq \frac{\log n}{200}}\leq \Pr\brac{Bin\brac{{\frac{\log n}{100},1-p}}\geq \frac{\log n}{200}}\leq \bfrac{1200e\log\log n}{\log^{1/2}n}^{\log n/200}=o(n^{-2}),
\]
and in general we have
\[
\Pr\brac{\exists\ 2\leq i\leq i_0:|S_i|\leq \frac{|S_{i-1}|\log n}{100}}=o(n^{-2}).
\]
\proofend

The total number of vertices added to $W$ in this way throughout the whole of Phase 2 is $O(n^{1/2}\log^3n)=o(n^{3/4})$. (As we see later, we try this process once for each small $C$.)

Let $t^*\leq i_0=\rdup{1.5\log n}$ denote the value of $t$ when we stop the growth of $T_0$. At this stage $T_0$ has leaves $\Gamma_i$, for $i = 1,\ldots,\nu$, each with a path of length at least $n_0$, (unless we have already successfully made a cycle). We now execute an In-Phase. This involves the construction of trees $T_i,i=1,2,\ldots \nu$. Assume that $P(\Gamma_i)$ is a path from $u_0$ to $v_i$. We start with $\Gamma_i$ and build $T_i$ in a similar way to $T_0$ except that here all paths generated end with $v_i$. This is done as follows: if a current NPD $\Gamma$ has $P(\Gamma)\in {\cal P}(u,v_i)$ then we consider adding an edge $(w,u)\in E(K_2)$ and deleting an edge $(w,x)\in \Gamma$. Thus our trees are grown by considering edges directed into the start vertex of each $P(\Gamma)$ rather than directed out of the end vertex. Some technical changes are necessary however.

We consider the construction of our $\nu$ In-Phase trees in two stages. First of all we grow the trees only enforcing condition {\bf C} (ii) of success and thus allow the formation of small
cycles and paths. We try to grow them to depth $t_0$. The growth of the $\nu$ trees can naturally be considered to occur simultaneously. Let $L_{i,\ell }$ denote the set of start vertices of the paths associated with the nodes at depth $\ell$ of the $i$'th tree, $i=1,2\ldots ,\nu, \ell = 0,1,\ldots,t_0$. Thus $L_{i,0}=\{ u_0\}$ for all $i$. We prove inductively that $L_{i,\ell}=L_{1,\ell}$ for all $i,\ell$. In fact if $L_{i,\ell}=L_{1,\ell}$ then the acceptable $E(K_2)$ edges have the same set of initial vertices and since all of the deleted edges are $E(K_2)$-edges (enforced by {\bf C} (ii)) we have $L_{i,\ell +1}=L_{1,\ell +1}$. The fact that $L_{i,\ell}=L_{1,\ell}$ gives us some control over the set of new vertices exposed in the In-Phase.

The probability that we succeed in constructing trees $T_1,T_2,\ldots T_{\nu}$ is, by the analysis of Lemma \ref{lem9}, $1-o(n^{-2})$. Note that the number of nodes in each tree is $O(\D_1^{i_0})=O( n^{1/2+o(1)})$.

We now consider the fact that in some of the trees some of the leaves may have been constructed in violation of {\bf C} (i). We imagine that we prune the trees $T_1,T_2,\ldots T_\nu$ by disallowing any node that was constructed in violation of {\bf C} (i). Let a tree be BAD if after pruning it has less than $\nu$ leaves and GOOD otherwise. Now an individual pruned tree has been constructed in the same manner as the tree $T_0$ obtained in the
Out-Phase. Thus
$$Pr(T_1 \mbox{ is BAD})=O(n^{-2+o(1)}) \text{ and }\Pr(\exists\ \text{ a BAD tree})=o(1).$$
Thus with probability 1-$O(n^{-2+o(1)})$ we end up with $\nu$ sets of $\nu$ paths, each of length at least $n/\log^{1/2}n$ where the $i$'th set of paths all terminate in $v_i$. The sets $in(v_i)$ have not yet been exposed by the process and hence
\[
\Pr(\text{no $E(K_2)$ edge closes one of these paths})  \leq  \brac{1-\frac{\n\log n}{100m}}^{\n}=o(n^{-2}).
\]
Consequently the probability that we fail to eliminate a particular small cycle $C$ after breaking an edge is $o(n^{-2})$ and so w.h.p. we remove all small cycles. Thus we have shown that at the end of Phase 2 we have a PD $\Pi^{*}$ in which the minimum cycle length is at least $n_0$.
\subsection{Phase 3. Patching $\Pi^{*}$  to a Hamilton cycle}
We begin this phase with a cycle cover of $O(\log^{1/2}n)$ cycles, each of size at least $n/\log^{1/2}n$. We use the edges of $E(C_3)$ to {\em patch} these cycles into a Hamilton cycle.
Suppose at some stage of the patching process we have cycles $L_1,L_2,\ldots,L_k,k\geq 2$ where $L_1$ is a largest cycle and $L_2$ is a second largest cycle. We look for edges $(i,j)\in L_1,(k,l)\in L_2$ such that $(a_i,b_l),(a_k,b_j)\in E(C_3)$ and then replace $L_1,L_2$ by the cycle  $L_1+L_2+(i,l)+(k,j)-(i,j)-(k,l)$. Given that the degrees of vertices in $C_3$ are at least $\log n/100$, the probability we cannot find two such edges is at most
\[
\brac{1-\bfrac{\log^2n}{10^4m}}^{((1-o(1))n/\log^{1/2}n)^2}\leq n^{-\e},\qquad \e=10^{-9}.
\]
This is true regardless of previous patches because all the added edges have one end in $L_2$ and the vertices of the second lexicographically largest cycle $L_2$ at one stage are disjoint from lexicographically largest cycle at any other stage, the former having been absorbed into $L_1$.

Thus the probability we fail to create a Hamilton cycle is $O(n^{-\e}\log^{1/2}n)=o(1)$. This completes the proof of Theorem \ref{th1}.
\section{Conclusion}
As already mentioned, we feel that the above analysis can be extended to get $k$ edge disjoint Hamilton cycles if we condition on minimum in- and out-degree at least $k$, as long as we have $\sim \frac12n\log n$ edges. This is not to say that there might not be unforeseen technicalities and so we will not formally claim this.

It is even more challenging to adapt the analysis to the case where we have a linear number of edges and minimum degree $k\geq 2$. This is where the focus of research on this problem should now be.

\section*{Appendix A: Proof of Lemma \ref{Om1}}\label{A}
\begin{proof}
Let
$$S = \Bigl\{ \bd \in [n]^n \,\Big|\sum_{1\leq j \leq n} d_j = m \text{ and }\forall j,\, d_j \geq 1\Bigr\}.$$
Fix $\boldsymbol{\xi} \in S$.   Then, by the definition of $\bx$ and $d_\bx$,
$$\Pr(d_\bx = \boldsymbol{\xi}) =\left( \frac{m!}{\xi_1! \xi_2! \ldots \xi_n! }\right)\bigg/\left( \sum_{x\in S} \frac{m!}{x_1! x_2! \ldots x_n!} \right).$$
On the other hand, if $\bZ=(Z_1,Z_2,\ldots,Z_n)$ where the $Z_i$ are independent copies of the truncated Poisson random variable $Z$, then
\begin{align*}
\Pr \left(\bZ = \boldsymbol{\xi} \; \bigg|\; \sum_{1\leq j \leq n} Z_j =  m\right)
&=\left( \prod_{1\leq j\leq n}{z^{\xi_j}\over (e^{z} -1) \xi_j!}\right)
\bigg/\left( \sum_{x\in S} \prod_{1\leq j\leq n}{z^{x_j}\over (e^{z} -1) x_j!}\right)\\
&=\left(   { (e^z-1)^{-n}z^{m}\over \xi_1! \xi_2! \ldots \xi_n!} \right)
\bigg/\left( \sum_{x\in S} {(e^{z}-1)^{-n} z^{m} \over x_1! x_2! \ldots
x_n!} \right)\\
&=\Pr(\bd_\bx=\boldsymbol{\xi}).
\end{align*}
\end{proof}
\section*{Appendix B: Proof of Claim \ref{cl1}}\label{B}
We repeat the claim:
\begin{claim}
Suppose that $\bx$ is a random member of $\Omega_1$ and that $|L_\bx|=k,|M_\bx|=\ell$. Then after $\ell$ $P$-switches and $k+2\ell$ random $L$-switches we obtain $\by\in \Okl{0,0}$ such that (i) \by\ is almost uniform in $\Okl{0,0}$ and (ii) $|\Okl{0,0}|= (1-o(1))|\Omega_1^*|$. (By almost uniform in (i) we mean the following: if \bz\ is any member of $\Okl{0,0}$, then $\Pr(\by=\bz)=(1+o(1))|\Okl{0,0}|^{-1}$.)
\end{claim}
We need to have  sharp estimates of the probability that $\sum_{1\leq j\leq n}Z_j',\sum_{1\leq j\leq n}Z_j''$ are close to their mean $m$. Let $\s$ be as in \eqref{2}. Section A of \cite{FP} proves
the following proposition:
\begin{proposition}\label{propx}
\begin{eqnarray}
\label{ll1}
\Pr\left(\sum_{j=1}^nZ_j'=m\right)&=&\frac{1}{\s\sqrt{2\p n}}(1+O(n^{-1}\s^{-2}))\\
\noalign{and, if in addition, $k=O(n^{1/2}\s)$}
\Pr\left(\sum_{j=2}^nZ_j'=m-k\right)&=&\frac{1}{\s\sqrt{2\p n}}\left(1+
O((k^2+1)n^{-1}\s^{-2})\right).\label{ll2}
\end{eqnarray}
\end{proposition}
Recall that $\D_\bx=\max_j(d_{\bx}^-(j)+d_{\bx}^+(j))$. Let
\begin{equation}\label{S1}
S_1=S_1(\bx)=\sum_{i\in[n]}d_{\bx}^+(i)(d_{\bx}^+(i)-1).
\end{equation}
We next bound the sizes of $L_\bx,M_\bx$.
\begin{lemma}\label{LM}
Suppose that \bx\ is drawn uniformly from $\Omega_1$ and that \bz\ denotes one of the sequences constructed as we transform \bx\ to \by\ as in Claim \ref{cl1}. Assuming that $ \tfrac25\log n\leq m/n\leq \log n$, then w.h.p. throughout the construction,
\begin{enumerate}[(a)]
\item $\D_\bz<\D_0=\log^2n$, with probability $1-(\log n)^{-\Omega(\log^2n)}$,
\item $|L_\bz|\leq 2e^2\log^4n$, with probability $1-(\log n)^{-\Omega(\log^3n)}$.
\item No two loops are adjacent in \bz, w.h.p.
\item No edge is repeated more than twice in \bz, w.h.p.
\item No two pairs of parallel edges share a vertex in \bz, w.h.p.
\item No loop lies on a vertex that is also on a parallel edge in \bz, w.h.p.
\item $|M_\bz|\leq e^2\log^4n$, with probability $1-(\log n)^{-\Omega(\log^4n)}$.
\item $\card{S_1-mz}\leq n^{2/3}$, with probability at least $1-(\log n)^{-\Omega(\log^3n)}$,
where  $S_1$ is given by \eqref{S1}.
\end{enumerate}
\end{lemma}
\begin{proof}
We first observe that the function $f(z)=ze^z/(e^z-1)$ is monotone increasing.
\beq{z}{
\frac{z}{\r}=1-e^{-z}\text{ and so }\r\geq \tfrac25 \log n\text{ implies that }z \sim\r\geq \tfrac25\log n.
}
(a) Using our assumptions on $m$ and \eqref{z} we see that $\s^2\sim z\sim\r$. Then,
We have
\mult{Delta}{
\Pr(d_\bx^+(1)\geq \D_0/2)\leq \frac{\Pr(Z_1'\geq \D_0/2)}{\Pr\brac{\sum_{j=1}^nZ_j'=m}}\leq
\frac{\sum_{j\geq \D_0/2}\frac{z^j}{j!(e^{z}-1)}}{(1-o(1))/(\s \sqrt{2\pi n})}\leq\\ n^{1/2}\bfrac{2e}{\D_0}^{\D_0/2}=(\log n)^{-\Omega(\log^2n)}.
}
This upper bound holds for \bz\ too, as the total degree of a vertex does not change.

(b) Suppose that $k=O(\log^{O(1)}n)$. Then, assuming
\begin{align*}
\E\brac{\binom{|L_\bx|}{k}}&\leq \binom{m}{k}\frac{\Pr(x_{2j-1}=x_{2j},j=1,2,\ldots,k)}{\Pr\brac{\sum_{j=1}^nZ_j'=m}\Pr\brac{\sum_{j=1}^nZ_j''=m}}\\
&\leq \bfrac{me}{k}^k\s^2n \bfrac{\D_0}{m}^k=\s^2n\bfrac{e\log^2n}{k}^k.
\end{align*}
Then for $t>0$ we have
\[
\Pr(|L_\bx|\geq t)=\Pr\brac{\binom{|L_\bx|}{k}\geq \binom{t}{k}}\leq \frac{\E\brac{\binom{|L_\bx|}{k}}}{\binom{t}{k}}\leq \s^2 n\bfrac{e\log^2n}{t}^k.
\]
Putting $t=k=\log^3n$ yields $|L_\bx|\leq \log^3n$ with the required probability. Now $|L_\bz|\leq |L_\bx|+2|M_\bx|$ and so (b) will follow, once we verify (g) below.

(d) For this and (c) and (e) we need only prove the result for \bx. This is because under the claimed circumstance, switches will preserve these properties. We rule out an edge appearing at least three times. Let $\cE_1$ be the event $\set{x_1=x_3=x_5=1}$ and let $\cE_2$ be the event $\set{x_2=x_4=x_6=2}$ and note that $\Pr(\cE_i)\leq (\D_0/m)^3$ for $i=1,2$.
\begin{align*}
\Pr(\neg (d))&\leq \binom{m}{3}n^2\cdot\Pr\brac{\cE_1\bigg|\sum_{j=1}^nZ_j'=m} \Pr\brac{\cE_2\bigg|\sum_{j=1}^nZ_j'=m}\\
&=\binom{m}{3}n^2\cdot\frac{\Pr\brac{\sum_{j=1}^nZ_j'=m\mid \cE_1}\Pr(\cE_1)}{\Pr\brac{\sum_{j=1}^nZ_j'=m}}\cdot\frac{\Pr\brac{\sum_{j=1}^nZ_j''=m\mid \cE_2}\Pr(\cE_2)}{\Pr\brac{\sum_{j=1}^nZ_j''=m}}\\
&\leq \frac{m^3n^2\D_0^6}{m^6}\cdot\frac{\Pr\brac{\sum_{j=1}^nZ_j'=m\mid \cE_1}} {\Pr\brac{\sum_{j=1}^nZ_j'=m}} \cdot\frac{\Pr\brac{\sum_{j=1}^nZ_j''=m\mid \cE_2}}{\Pr\brac{\sum_{j=1}^nZ_j''=m}}.
\end{align*}
Now,
\begin{align*}
\Pr\brac{\sum_{j=1}^nZ_j'=m\mid \cE_1}&=\sum_{k=3}^{\D_0} \Pr\brac{\sum_{j=2}^nZ_j'=m-k\mid \cE_1,Z_1'=k}+(\log n)^{-\Omega(\log^2n)}\\
&=\sum_{k=3}^{\D_0} \Pr\brac{\sum_{j=2}^nZ_j'=m-k}+(\log n)^{-\Omega(\log^2n)}\\
&\sim \Pr\brac{\sum_{j=1}^nZ_j'=m},\quad\text{using \eqref{ll1} and \eqref{ll2}}.
\end{align*}
Thus
\[
\Pr(\neg (d))\leq \frac{2m^3n^2\D_0^6}{n^6}=o(1).
\]
(e) In a similar vein,
\[
\Pr(\neg (e))\leq(1+o(1))\binom{m}{4}n^3\bfrac{\D_0}{m}^{8}=o(1).
\]

(c)
In a similar vein,
\[
\Pr(\neg (c))\leq(1+o(1))\binom{m}{2}n\bfrac{\D_0}{m}^{4}=o(1).
\]
(f)
In a similar vein,
\[
\Pr(\neg (f))\leq(1+o(1))\binom{m}{3}n^2\bfrac{\D_0}{n}^6=o(1).
\]
(g) Given (d), we can assume that $M_\bx$ consists of pairs of repeated edges. Then, as in (b),
\begin{align*}
\E\brac{\binom{|M_\bx|}{k}}&\leq \binom{m}{2k}\frac{\Pr(x_{2j-1}=x_1,x_{2j}=x_2,j=2,3,\ldots,k)}{\Pr\brac{\sum_{j=1}^nZ_j'=m}\Pr\brac{\sum_{j=1}^nZ_j''=m}}\\
&\leq n^{2k} \binom{m}{2k}\frac{\Pr(x_{2j-1}=j-1,x_{2j}=j,j=1,2,\ldots,k)}{\Pr\brac{\sum_{j=1}^nZ_j'=m}\Pr\brac{\sum_{j=1}^nZ_j''=m}}\\
&\leq 4\bfrac{me}{2k}^{2k}n\log^2 n\bfrac{\D_0}{m}^{2k}=4n^{1/2}\log n\bfrac{e\log^2n}{2k}^{2k}.
\end{align*}
And so for $t>0$,
\[
\Pr(|M_\bx|\geq t)=\Pr\brac{\binom{|M_\bx|}{k}\geq \binom{t}{k}}\leq \frac{\E\brac{\binom{|M_\bx|}{k}}}{\binom{t}{k}}\leq 4n\log^2 n\bfrac{e\log^2n}{4t^{1/2}}^{2k}.
\]
Putting $t=k=\log^4n$ yields (f).

(h) Let $X_j=Z_j'(Z_j'-1)-\E(Z_j'(Z_j'-1))=Z_j'(Z_j'-1)-mz/n$ for $j\in[n]$. Let $A=\log^6n$. 
\[
\Pr(X_j\geq A)\leq \Pr(Z_j'\geq \log^3n)=\sum_{j\geq \log^3n}\frac{z^j}{j!(e^{z}-1)}\leq \bfrac{e}{\log^2n}^{\log^3n}.
\]
Putting $X_j'=\min\set{X_j,A}$ we see that
\begin{align*}
\Pr(|S_1-mz|\geq n^{2/3})&= \Pr\brac{\card{\sum_{j=1}^nZ_j'(Z_j'-1)-mz}\geq n^{2/3}\ \bigg|\sum_{j=1}^nZ_j'=m}\\
&\leb (n^{1/2}\log n)\Pr\brac{\card{\sum_{j=1}^nZ_j'(Z_j'-1)-mz}\geq n^{2/3}}\\
&=(n^{1/2}\log n)\Pr\brac{\card{\sum_{j=1}^nX_j}\geq n^{2/3}}\\
&\leq (n^{1/2}\log n)\Pr\brac{\card{\sum_{j=1}^nX_j'}\geq n^{2/3}}+\Pr(\exists j:X_j\geq A).
\end{align*}
Now $-\log^2n\leq X_j'\leq A$ and so by Hoeffding's theorem
\[
\Pr(|S_1(\bx)-mz|\geq n^{2/3}/2)\leq (n^{1/2}\log n)\exp\set{-\frac{n^{4/3}}{2n(A^2+\log^2n)}}+n\bfrac{e}{\log^2n}^{\log^3n}.
\]
Now $|S_1(\bz)-S_1(\bx)|=O(\log^{10}n)$ and this completes the proof of (h).
\end{proof}
Recall the definition of $\Okl{k,\ell}=\set{\bx\in \Omega_1:|L_\bx|=k,|M_\bx|=\ell\text{ and }\cP}$ where we define $\cP=\cP_1\cup\cP_2$ where $\cP_1$ is the property that the high probability events described in Lemma \ref{LM}(a),(b),(g),(h) hold and $\cP_2$ is similar, with reference to (c),(d),(e),(f). We note that $\Pr(\neg\cP_1=O((\log n)^{-\log^2n})$.

We define two bipartite graphs $\G_1, \G_2$. Firstly we define $\G_1=\G_1(k,\ell)$ with bipartition $\Okl{k,\ell},\Okl{k+2,\ell-1}$. Join $\bx\in \Okl{k,\ell}$ to $\by\in \Okl{k+2,\ell-1}$ by an edge in $\G_1$ if   \by\ can be  obtained from \bx\ by a $P$-switch. Thus there are $i,j\in M_\bx$ such that (i) $y_{2i}=x_{2i-1}$ and $y_{2j-1}=x_{2i}$ and (ii) $y_t=x_t$ for $t\notin \set{2i,2j-1}$.
\begin{lemma}\label{switch1}
Let $d_1$ denote degree in $\G_1$.
\begin{enumerate}[(a)]
\item $\bx\in  \Okl{k,\ell}$ implies that $d_1(\bx)=\ell$.
\item $\by\in \Okl{k+2,\ell-1}$ implies that $d_1(\by)=(k+1)(k+2)$.
\end{enumerate}
\end{lemma}
\begin{proof}
(a) Each  pair of parallel edges in \bx\ yields two loops that are not adjacent because Lemma \ref{LM} (d) holds.\\
(b) Each pair of loops in \by\  yields two pairs of parallel edges, but does not create a triple edge because Lemma \ref{LM} (c) holds.
\end{proof}

We define $\G_2=\G_2(k)$ with bipartition $\Okl{k,0},\Okl{k-1,0}$. Join $\bx\in \Okl{k,0}$ to $\by\in \Okl{k-1,0}$ by an edge in $\G_2$ if \by\ can be obtained from \bx\ by an $L$-switch. Thus there are  $i\in L_\bx$ and $j\notin L_\bx$ such that (i)  $y_{2i}=x_{2j-1},y_{2j-1}=x_{2i}$ and (ii) $y_t=x_t$ for $t\notin \set{2i,2j-1}$.

\begin{lemma}\label{switch2} Let $d_2$ denote degree in $\G_2$.
\begin{enumerate}[(a)]
\item $\bx\in  \Okl{k,0}$ implies that $k(m-k+1-\D_0)\leq d_2(\bx)\leq k(m-k+1)$.
\item $\by\in \Okl{k-1,0}$ implies that $S_1(\by)-2\D_0\leq d_2(\by)\leq S_1(\by)$.
\end{enumerate}
\end{lemma}
\begin{proof}
(a) The upper bound is clear, we choose a loop and another non-loop and replace them by two edges incident with vertex containing the loop. For the lower bound, we subtract those $j$ that would lead to a parallel edge containing vertex $x_{2i-1}$.\\
(b) The degree of $\by\in \Okl{k-1,0}$ in $\G_2(k)$ is between $S_1(\by)-2\D_0$ and $S_1(\by)$. The upper bound is clear, we choose a vertex $v$ of \by\ and two out-neighbours $w_1,w_2$ and replace $(v,w_1),(v,w_2)$ by $(v,v),(w_1,w_2)$. For the lower bound, we subtract those choices that would create a parallel edge.
\end{proof}

\begin{corollary}\label{cor1}
Assume that $m\sim \tfrac12n\log n$ and that $k,\ell\leq 20e^2\log^2n$, then
\[
|\Okl{k,\ell}|=\brac{1+O\bfrac{\log^4n}{n}}\frac{\th^{k+2\ell}}{k!\ell!}|\Okl{0,0}|,
\]
where $\th=\E(S_1(\bx))/m$, and by Lemma \ref{LM} (h),  $\th\sim z$. (Here the expectation of $S_1(\bx)$ is over a random \bx\ from $\Omega_1$. Although it might seem necessary to define it over a random \bx\ in $\Okl{k,\ell}$, $S_1$ changes by a small enough amount as we make changes and so it suffices to stick with the current definition.)
\end{corollary}
\begin{proof}
Applying Lemma \ref{switch1} repeatedly, we see that for $k,\ell\leq 20e^2\log^2n$, we have
\beq{O1}{
|\Okl{k,\ell}|=\frac{(k+2\ell)!}{k!\ell!}|\Okl{k+2\ell,0}|.
}
Then applying Lemma \ref{switch2} repeatedly, we see that
\beq{O2}{
|\Okl{k+2\ell,0}|=\brac{1+O\brac{(k+2\ell)\D_0\brac{\frac{1}{\E(S_1)}+\frac{1}{m}}}}
\frac{(\E(S_1))^{k+2\ell}}{(k+2\ell)!m^{k+2\ell}}|\Okl{0,0}|.
}
Thus,
\[
|\Okl{k,\ell}|=\brac{1+O\bfrac{\log^4n}{n}}\frac{(\E(S_1))^{k+2\ell}}{k!\ell!m^{k+2\ell}}|\Okl{0,0}|.
\]
\end{proof}
We can now complete the proof of Claim \ref{cl1} as follows: we can write,
\[
 \Omega_1=\bigcup_{k,\ell\leq 20e^2\log^2n}\Okl{k,\ell}\cup X,
\]
where
\[
X=\set{\bx\in \Omega_1:(\max\set{k,\ell}>20e^2\log^2n)\vee ((k,\ell) \le 20e^2\log^2n\text{ and }\neg\cP)}.
\]
So,
\[
\sum_{k,\ell\leq 20e^2\log^2n}|\Okl{k,\ell}|\leq |\Omega_1|
\leq \sum_{k,\ell\leq 20e^2\log^2n}|\Okl{k,\ell}|+|\set{\bx\in \Omega_1:(\max\set{k,\ell}>20e^2\log^2n)\vee \neg\cP}|.
\]
Therefore, from Lemma \ref{LM} and Corollary \ref{cor1},
\mults{
\sum_{k,\ell\leq 20e^2\log^2n}\frac{\th^{k+2\ell}}{k!\ell!}|\Okl{0,0}|\leq\brac{1+O\bfrac{\log^4n}{n}} |\Omega_1|\leq\\
\brac{1+O\bfrac{\log^4n}{n}} \sum_{k,\ell\leq 20e^2\log^2n}\frac{\th^{k+2\ell}}{k!\ell!}|\Okl{0,0}|+|\Omega_1|(\log n)^{-\Omega(\log^2n)}.
}
Now $\th=O(\log n)$ and so $\sum_{k,\ell\leq 20e^2\log^2n}\frac{\th^{k+2\ell}}{k!\ell!}=e^{\th(\th+1)}(1-O((\log n)^{-\Omega(\log^2n)}))$.

Now
\[
\Okl{0,0}\subseteq\Omega_1^*=\set{\bx\in \Omega_1:L_\bx=M_\bx=\emptyset}\subseteq \Okl{0,0}\cup \set{\bx\in \Omega_1:\neg\cP_1}.
\]
Therefore,
\[
|\Okl{0,0}|\leq |\Omega_1^*|\leq |\Okl{0,0}|+|\Omega_1|(\log n)^{-\log^2n}.
\]
Because $|\Okl{0,0}|\sim e^{-\th(\th+1)}|\Omega_1|$ and $\th=O(\log n)$, this implies that
\beq{f}{
|\Omega_1^*|\sim |\Okl{0,0}|\sim e^{-\th(\th+1)}|\Omega_1|
}
This completes the proof of Claim \ref{cl1}(ii). The above analysis also proves Claim \ref{cl1}(i). It shows that choosing random switches leads to a near uniform member of $\Okl{0,0}$, which is almost all of $\Omega_1^*$. Suppose we start with \bx\ chosen uniformly from $\Omega_1$ and that $\bx\in \Okl{k,\ell}$ and that we construct a sequence $\by_1,\by_2,\ldots,\by_{k+2\ell}$ where $\by_{k+2\ell}\in \Okl{0,0}$. Then for any $\by\in \Okl{0,0}$, we have
\begin{align*}
\Pr(\by_{k+2\ell}=\by)&=\brac{1+O\bfrac{\log^4n}{m}}\frac{1}{|\Omega_1|}\prod_{i=1}^\ell\frac{(k+2i-1)(k+2i)}{i}\prod_{j=1}^{k+2\ell}\frac{\E(S_1)}{jm}\\
&=\brac{1+O\bfrac{\log^4n}{m}}\frac{1}{|\Omega_1|}\frac{\E(S_1)^{k+2\ell}}{k!\ell!m^{k+2\ell}}.
\end{align*}
The RHS is independent of the choice of \bx.  This completes the proof of Claim \ref{cl1}.

\section*{Appendix C: Proof of Theorem \ref{th3}}
We observe that for an $y>0$,
\beq{trick}{
|\Omega_1|=(m![x^{m}](e^x-1)^n)^2=\brac{\frac{m!(e^y-1)^{n}}{y^{m}}[x^{m}]\bfrac{e^{xy}-1}{e^y-1}^n}^2.
}
Notice that $p_y(x)=\frac{e^{xy}-1}{e^y-1}$ is the probability generating function of the random variable $Z(y)$. It then follows from \eqref{trick} that
\beq{tr1}{
|\Omega_1|=\brac{\frac{m!(e^y-1)^{n}}{y^{m}}\Pr(X_1+X_2+\cdots+X_n=m)}^2
}
where $X_1,X_2,\ldots,X_n$ are independent copies of $Z(y)$. We choose $y=z$ so that the probability in \eqref{tr1} is large. To finish we use \eqref{sizeD} and \eqref{ll1} and \eqref{f}.
\proofend

\begin{thebibliography}{99}
\bibitem{A}  M. Anastos, Packing Hamilton cycles in cores of random graphs, arxiv2107.03527
%
\bibitem{AF2} M. Anastos and A.M. Frieze, Hamilton cycles in random graphs with minimum degree at least 3: an improved analysis, {\em
Random Structures and Algorithms} 57 (2020) 865-878.
%
\bibitem{AFP} J. Aronson, A.M. Frieze and B. Pittel, Maximum matchings in sparse random graphs: Karp-Sipser revisited, {\em Random Structures and Algorithms} 12 (1998) 111-178.
%
\bibitem{BCFF} B. Bollob\'as, C. Cooper, T. Fenner and A.M. Frieze, On Hamilton cycles in sparse random graphs with minimum degree at least $k$, {\em  Journal of Graph Theory} 34 (2000) 42-59.
%
\bibitem{BFF} B. Bollob\'as, T. Fenner and A. M. Frieze, Hamilton cycles in random graphs with minimal degree at least $k$, {\em In A tribute to Paul Erd\H{o}s, Edited by A. Baker, B. Bollobás, A. Hajnal}, Cambridge University Press (1990) 59 - 96.
%
\bibitem{CF1} C. Cooper and A.M. Frieze, Hamilton cycles in a class of random directed graphs, {\em Journal of Combinatorial Theory} B 62 (1994) 151-163.
%
\bibitem{Durret} R. Durrett, Probability: Theory and Examples, 4th Edition,
{\em Cambridge University Press} 2010.
%
\bibitem{DF} M.E.~Dyer and A.M.~Frieze, {\em On patching algorithms for random asymmetric traveling salesman problems}, Mathematical Programming 46  (1990) 361--378.
%
\bibitem{F1} A.M. Frieze, An algorithm for finding hamilton cycles in random digraphs, {\em Journal of Algorithms} 9 (1988) 181-204.
%
\bibitem{FKR}  A.M.~Frieze, R.M.~Karp and B.~Reed, {\em When is the assignment bound asymptotically tight for the asymmetric traveling-salesman problem?}, SIAM Journal on Computing 24 (1995) 484--493.
%
\bibitem{FP} A.M. Frieze and B. Pittel, Perfect matchings in random graphs with prescribed minimal degree, {\em Trends in Mathematics, Birkhauser Verlag, Basel} (2004) 95-132.
%
\bibitem{FS} A.M. Frieze and G. Sorkin, The probabilistic relationship between the assignment and asymmetric traveling salesman problems, {\em SIAM Journal on Computing} 36 (2007) 1435-1452.
%
\bibitem{K} R.M.~Karp, {\em A patching algorithm for  the non-symmetric traveling salesman problem}, SIAM Journal on Computing 8 (1979) 561--573.
%
\bibitem{KS} R.M.~Karp and J.M.~Steele, \emph{Probabilistic analysis of heuristics}, in The traveling salesman problem: a guided tour of combinatorial optimization, E.L.~Lawler, J.K.~Lenstra, A.H.G.~Rinnooy Kan and D.B.~Shmoys Eds. (1985) 181--206.
%
\bibitem{McD} C. McDiarmid, Clutter percolation and random graphs, {\em Mathematical Programming Studies} 13 (1980) 17-25.
%
\bibitem{RW1} R. Robinson and N. Wormald, Almost all cubic graphs are Hamiltonian, {\em Random Structures and Algorithms} 3 (1992) 117-125.
\end{thebibliography}
\end{document}